\documentclass[preprint,3p]{elsarticle}
\journal{Journal of Approximation Theory}
\usepackage{amsmath,amssymb,latexsym,color}
\usepackage{amsthm}
\usepackage{dsfont}
\usepackage{pdfsync,color}

\newtheorem{theorem}{Theorem}
\newtheorem{lemma}[theorem]{Lemma}

\newtheorem{proposition}{Proposition}
\newtheorem*{ack*}{Acknowledgements}
\newcommand{\CC}{\mathds{C}}

\newcommand{\NN}{\mathds{N}}
\newcommand{\PP}{\mathds{P}}
\newcommand{\RR}{\mathds{R}}
\newcommand{\ZZ}{\mathds{Z}}
\newcommand{\LL}{\mathcal{L}}

\newcommand{\dsty}{\displaystyle}

\newcommand{\unifn}{\;\; {\mathop{\mbox{\Large $\rightrightarrows$}}_{n\to\infty}}\;\;}
\newcommand{\weak}{\;\;{\stackrel{\rm *}{\longrightarrow}}\;\;}

\begin{document}

\begin{frontmatter}

\title{Differential orthogonality: Laguerre and Hermite cases with applications.}

\author[label1]{J. Borrego--Morell\fnref{grant1,grant2}}
\address[label1]{UNESP--Universidade Estadual Paulista, IBILCE, Brazil.}
\ead{jbmorell@gmail.com}

\author[label2]{H. Pijeira--Cabrera\fnref{grant2}}
\address[label2]{Universidad Carlos III de Madrid, Spain.}
\ead{hpijeira@math.uc3m.es}

\fntext[grant1]{Research partially supported by FAPESP of Brazil.}

\fntext[grant2]{Research partially supported  by Ministerio de Econom\'{\i}a y Competitividad of Spain, under grant MTM2012-36732-C03-01.}

\begin{abstract}
Let $\mu$ be a finite positive Borel measure supported  on $\RR$, $\LL[f] =xf''+(\alpha+1-x)f'$ with $\alpha>-1$, or $\LL[f] =\frac{1}{2}f''-xf'$, and $m$ a natural number. We study algebraic, analytic and
asymptotic properties of the sequence of monic polynomials $\{Q_n\}_{n>m}$ that satisfy
the orthogonality relations
$$
\int \LL[Q_n](x)x^kd\mu(x) = 0 \quad \mbox{for all} \quad  0 \leq k \leq  n-1.
$$
We also provide a fluid dynamics model for the zeros of these polynomials.

\end{abstract}


\end{frontmatter}

\section{Introduction}
\label{SecIntro}

Orthogonal polynomials with respect to a differential operator  were introduced in \cite{ApLoMa02}  as a  generalization of the notion of orthogonal polynomials. Analytic and algebraic properties of these classes of polynomials have been considered for some classes of first order differential operators in \cite{BePiMaUr11, PiBeUr10},  for a Jacobi differential operator in \cite{BorPij12}, and for differential operators of arbitrary order with polynomials coefficients in \cite{Bor12}.  In this paper,  we consider orthogonal polynomials with respect to a Laguerre or Hermite operator and a positive Borel measure  $\mu $   with unbounded support on $\RR$.

We denote by $\LL_L$ the Laguerre and by  $\LL_H$ the Hermite differential operators on the linear space $\PP$ of all polynomials, i.e. for all $ f \in \PP$ and $\alpha>-1$
\begin{eqnarray}
\label{Lag_DO}
\LL_L[f] & = & xf^{\prime \prime} + (1+\alpha -x) f^{\prime}= x^{-\alpha}\,e^{x} \left(x^{\alpha+1}\,e^{-x}\, f^{\prime}\right)^{\prime},  \\
\label{Her_DO}
\LL_H[f] & = &\frac{1}{2} f^{\prime \prime} -xf^{\prime}=\frac{1}{2} e^{x^2} \left(e^{-x^2}\, f^{\prime}\right)^{\prime}.
\end{eqnarray}

Each one of these second order differential operators has  a system of monic polynomials which are eigenfunctions of the operator and orthogonal with respect to a measure.  Let  $\{L^{\alpha}_n\}_{n=0}^{\infty}$ be the monic Laguerre polynomials with  $\alpha > -1$ and  $\{H_n\}_{n=0}^{\infty}$  the monic Hermite polynomials, then
$$
\begin{array}{cccc}
 \langle L^{\alpha}_n, L^{\alpha}_m \rangle_{L} & = & \dsty \int  L^{\alpha}_n(x)  L^{\alpha}_m(x) dw_{L}^{\alpha}(x)  & \left\{ \begin{array}{ccc}
                                                                           =0 & \mbox{if} & n \neq m, \\
                                                                           \neq 0 & \mbox{if}  & n=m, \\
                                                                         \end{array}\right. \\
 \langle H_n, H_m \rangle_{H} & = & \dsty \int  H_n(x)  H_m(x) dw_{H}(x) & \left\{ \begin{array}{ccc}
                                                                           =0 & \mbox{if} & n \neq m, \\
                                                                           \neq 0 & \mbox{if}  & n=m, \\
                                                                         \end{array}\right.
\end{array}$$
where $dw_{L}^{\alpha}(x)= x^{\alpha}\,e^{-x}dx$ and $dw_{H}(x)= e^{-x^2}dx$. In addition,
\begin{equation}\label{eigen_poly}
\LL_L[L^{\alpha}_n]=-n L^{\alpha}_n \quad \mbox{and} \quad \LL_H[H_n]=-n H_n.
\end{equation}

To unify the approach, we  will denote   by $\LL$  the Laguerre or Hermite differential operator ($\LL_L$ or $\LL_H$) in the sequel,  by  $dw$  the Laguerre or Hermite  measure ($dw_{L}^{\alpha}$ or $dw_{H}$), by $L_n$ the $n$th Laguerre or Hermite monic orthogonal polynomial ($L^{\alpha}_n$ or $H_n$) and by $\Delta$ the set $\RR_+$ or $\RR$, respectively. We will  refer to one  or the other depending on the case  we  are solving.

Let  $\mu$ be a finite positive Borel measure, supported  on $\Delta$  and  $\{P_n\}_{n=0}^{\infty}$ the corresponding  system of monic orthogonal polynomials, i.e.
\begin{equation}\label{OrthPoly_mu}
 \langle P_n,P_k \rangle_{\mu}=\int  P_n(x) P_k(x) d\mu(x) \left\{ \begin{array}{ll}
    \not =0 & \mbox{ if } n=k, \\
     = 0 & \mbox{ if }  n \neq k.  \\
   \end{array}\right.
\end{equation}

We say that  $Q_n$ is the $n$th monic orthogonal polynomial with respect to the pair  $(\LL, \mu)$  if $Q_n$ has degree $n$
 and
\begin{equation}\label{OrthDiff_01}
    \int \, \LL[Q_n] (x) \; x^k d\mu(x)=0 \quad
    \mbox{for all} \quad 0 \leq k \leq n-1,
\end{equation}
or, equivalently,
\begin{equation}\label{OrthDiff_02}
 \LL[Q_n]=\lambda_n \,P_n,
\end{equation}
where $\lambda_n=-n$.

 It was shown in \cite[\S2]{BorPij12} that  it is not always possible to guarantee the existence of a system of polynomials $\{Q_n\}_{n\in \ZZ_+}$ orthogonal with respect to the pair $(\LL^{(\alpha,\beta)},\mu)$, where $\LL^{(\alpha,\beta)}$ is the Jacobi differential operator and $\mu$ an arbitrary positive finite Borel measure. As will be shown later (cf. Propositions \ref{(Cor1)LH}  and \ref{Theo_EU}), a similar situation occurs for the case of Laguerre and Hermite operators.  Let   $m \in \NN$ be fixed, a fundamental role in the existence of infinite sequences of  polynomials $\{Q_n\}_{n>m}$ orthogonal with respect to the pair $(\LL,\mu)$ is  played by the class   $\mathcal{P}_m(\Delta)$  defined as the family of  finite positive Borel measures  $\mu$   supported on  $\Delta$  for which  there exist  a   polynomial $\rho$ of degree $m$, such that $\dsty \mu= \left(\rho \right)^{-1} w$.

If $\mu\in \mathcal{P}_m(\Delta)$ it is not difficult to see that if $n>m$, then
\begin{eqnarray}\label{PnLn}
   P_n(z) &=&  \sum_{k=0}^{m} b_{n,n-k} \;  L_{n-k}(z),\quad  b_{n,n-k}=\frac{1}{\tau_{n-k}} \int P_n(x) L_{n-k}(x) dw(x), \\
   \label{Lag-Herm_Norm2}
 \tau_n &=&\|L_n\|_{w}^2=\int L_n^2(x) dw(x)=\left\{\begin{array}{rl}
              n! \; \Gamma(n+\alpha+1) &, \mu\in\mathcal{P}_m(\RR_+), \\
              n! \sqrt{\pi}2^{-n}&, \mu\in\mathcal{P}_m(\RR),
            \end{array} \right.
\end{eqnarray}
and  from \eqref{OrthDiff_02} we obtain that   the monic polynomial of degree $n$, for $n>m$ defined by the formula
\begin{eqnarray}\label{Qhat_Def}
   \widehat{Q}_n(z) &=&  \sum_{k=0}^{m} \frac{\lambda_n}{\lambda_{n-k}}b_{n,n-k} \;  L_{n-k}(z),
\end{eqnarray}
is orthogonal with respect to $(\LL,\mu)$.

Notice that from the equivalence between relations   \eqref{OrthDiff_01}  and \eqref{OrthDiff_02}, the polynomial  $\widehat{Q}_n+c, c\in \CC$,  is orthogonal with respect to $(\LL,\mu)$ so that we do not have a unique monic orthogonal polynomial of degree $n$. We had a similar situation when we studied the orthogonality with respect to a Jacobi operator. A natural way  to define a unique sequence would be to consider  a sequence of complex numbers $\{\zeta_n\}_{n=m+1}^{\infty}$, and define  the sequence $\{Q_n\}_{n=m+1}^{\infty}$ satisfying  \eqref{OrthDiff_01}, as  the polynomial  solution of the initial value problem

\begin{equation}\label{IVP_nLH}\left\{\begin{array}{rcl}

                              \LL[y] & = & \lambda_n \,P_n, \quad n > m,\\

                               y(\zeta_n) &=& 0.

                             \end{array}\right.
\end{equation}
We say that $\{Q_n\}_{n=m+1}^{\infty}$ is the sequence of monic orthogonal polynomials with respect to the pair  $(\LL, \mu)$ such that  $Q_n(\zeta_n)=0$.

Notice  that the initial value problem \eqref{IVP_nLH} has the unique polynomial solution
\begin{equation}\label{(13)LH}
y(z)=Q_n(z)= \widehat{Q}_n(z) - \widehat{Q}_n(\zeta_n).
\end{equation}

In this paper, we study some analytic and algebraic properties of the sequence of orthogonal polynomials with respect to a Laguerre or Hermite differential operator. In order to study the asymptotic properties of the sequence of polynomials we shall normalize them with an adequate parameter.

Let $x_{n}$ be  the   modulus of the largest zero  of the $n$th orthogonal polynomial with respect to  $\mu$ (or $w$), from \cite[Lemma 11  with $\lambda=2$]{Rakh81}  for the Hermite case and  \cite[Coroll. (p. 191) with $\gamma=1$]{Rakh81}  for the Laguerre case,
we get
\begin{eqnarray}\label{asympzero}
\dsty \lim_{n\to\infty}c_n^{-1}x_n=1,
\end{eqnarray}
where $c_n$ is usually called Mhaskar-Rakhmanov-Saff constant, here with the closed expression
\begin{equation}\label{LargestZ-LH}
   c_{n}= \left\{ \begin{array}{rllcl}
                      4n & , & \mu \in \mathcal{P}_m(\RR_+) & \mbox{or}&  w(x)=x^{\alpha}e^{-x},x>0, \\
                      \sqrt{2n} & ,  &  \mu \in \mathcal{P}_m(\RR)& \mbox{or}&  w(x)= e^{-x^2},x\in \RR.
                    \end{array}
   \right.
\end{equation}
Throughout  this paper we denote the functions  $\varphi(z)=z+\sqrt{z^2-1}$ and  $\psi(z)=2z-1+2\sqrt{z(z-1)}$,  where the branch of each root is  selected from the condition $ \varphi(\infty)=\infty$ and $ \psi(\infty)=\infty $, respectively. Let $\Delta_c$ be the interval  $[0,1]$ in the Laguerre case and $[-1,1]$ in the Hermite case.   Let  $\mathfrak{P}_n(z)=c_n^{-n}P_n(c_nz)$  be  the normalized monic orthogonal polynomials with respect to a measure $\mu\in\mathcal{P}_m(\Delta)$.

To each generic polynomial $q_{n}$, let  $\dsty \mu_n={n^{-1}}\sum_{q_{n}(\omega)=0}\delta_{\omega}$ be the normalized root counting measure,  where $\delta_{\omega}$  is the Dirac measure  with mass $1$ at the point $\omega$. From \cite[Ths. 4 \& 4']{Rakh81} we find that the limit distribution $\nu_w$ of the zero counting measure of the normalized Laguerre  and Hermite polynomials is
$$
d\nu_w(t)=\begin{cases}
\dsty {2}{\pi^{-1}}\sqrt{\frac{1-t}{t}}dt,\quad t\in[0,1]& \text{Laguerre case,}\\
\dsty{2}{\pi^{-1}}\sqrt{1-t^2}dt,\quad t\in[-1,1] &  \text{Hermite case.}\\
\end{cases}
$$

From   \cite[Chs. III \& IV]{safftotik97} we have that
\begin{eqnarray}
\label{Pnrootc} \dsty \lim_{n  \rightarrow
\infty}\left|\mathfrak{P}_n(z)\right|^{\frac{1}{n}}&=&
\begin{cases}
 \frac{1}{e }\,\left|\psi(z)\right|\, e^{2\Re [1/\varphi(z)]}  & \mu \in \mathcal{P}_m(\RR_+),\\
 \frac{1}{2 \sqrt{e}} \,\left|\varphi(z)\right| \, e^{\Re{[z/\varphi(z)]}} &  \mu \in \mathcal{P}_m(\RR),
\end{cases}
\end{eqnarray}
uniformly on compact subsets $K \subset \CC \setminus \Delta_c$.

We are interested in asymptotic properties of the normalized monic  orthogonal polynomials with respect to a pair $(\LL,\mu)$ defined by
\begin{equation}\label{(13)LHnor}
\mathfrak{Q}_n(z)=\mathfrak{\widehat{Q}}_n(z)-\mathfrak{\widehat{Q}}_n(\zeta_n),
\end{equation}
where $\mathfrak{\widehat{Q}}_n(z)= c_n^{-n}\;\widehat{Q}_n(c_nz)$. For these polynomials we prove the followings results

\begin{theorem}\label{countingmeasureconv}

Let  $\mu \in \mathcal{P}_m(\Delta)$, where $m \in \NN$. Then:

\begin{itemize}

\item [a)] If $\nu_n,\sigma_n$ denote the root counting measure of ${\widehat{\mathfrak{Q}}}_n$ and ${\widehat{\mathfrak{Q}}}^{\prime}_n$ respectively then $\dsty \nu_n \weak \nu_{w}$ and $\dsty \sigma_n \weak \nu_{w}$
in the weak star sense.

\item [b)] The set of accumulation points of the zeros of $\left \{\mathfrak{\widehat{Q}}_n\right \}_{n=m+1}^{\infty}$ is $\Delta_c$.
\end{itemize}
\end{theorem}

\begin{theorem}\label{Th2}
Let $m \in \NN$, $\mu \in \mathcal{P}_m(\Delta)$. Then, for every compact subset $K$ of $\CC \setminus \Delta_c$ we have uniformly
\begin{eqnarray}\label{StrQ}
\lim_{n \to \infty}\frac{\mathfrak{P}_n(z)}{\widehat{\mathfrak{Q}}_n(z)}&=&\begin{cases}
\dsty 1 \quad & \mu \in \mathcal{P}_m(\RR_+)\\
\dsty 1 &  \mu \in \mathcal{P}_m(\RR)\\
\end{cases}
\\
\label{nrootcon}\dsty \lim_{n\to \infty}\left|\widehat{\mathfrak{Q}}_{n}(z)\right|^{\frac{1}{n}}\,&=&
\begin{cases}
 \frac{1}{e }\,\left|\psi(z)\right|\, e^{2\Re [1/\varphi(z)]}  & \mu \in \mathcal{P}_m(\RR_+),\\
 \frac{1}{2 \sqrt{e}} \,\left|\varphi(z)\right| \, e^{\Re{[z/\varphi(z)]}} &  \mu \in \mathcal{P}_m(\RR).\\
\end{cases}
\end{eqnarray}
\end{theorem}

The following result shows that  the set of accumulation points of the zeros of the sequence of normalized polynomials, orthogonal with respect to $(\LL,\mu)$ is contained in a curve.

\begin{theorem}\label{zeroloc}
Let $m \in \NN $ and $\mu \in \mathcal{P}_m(\Delta)$. If $\{\zeta_n\}_{n=m+1}^{\infty}$ is a sequence of complex numbers  with limit  $\zeta \in \CC \setminus \Delta_c$. Then:
\begin{itemize}

\item [a)] The accumulation points of zeros of the sequence $\{\mathfrak{Q}_{n}\}_{n=m+1}^{\infty}$ such that $\mathfrak{Q}_{n}(\zeta_n)=0$ are located on the set $E=\mathcal{E}(\zeta) \bigcup \Delta_c$, where $\mathcal{E}(\zeta)$ is the curve
\begin{equation}
    \mathcal{E}(\zeta):= \dsty\{z\in \CC:  \Psi(z)=\Psi(\zeta)\},
\end{equation}
$\Psi(z)=\left|\psi(z)\right|e^{2\Re [1/\varphi(z)]}$ for $\mu \in \mathcal{P}_m(\RR_+)$, and $\dsty \Psi(z)=\left|\varphi(z)\right|e^{\Re{[z/\varphi(z)]}}$ for $\mu \in \mathcal{P}_m(\RR)$.

\item [b)] If  $\dsty  \mathfrak{d}(\zeta)=\inf_{x\in\Delta_c}|\zeta-x|>2$ then $E=\mathcal{E}(\zeta)$ and for $n$ sufficiently large are simple.

\end{itemize}
\end{theorem}

The relative asymptotic behavior between the sequences of polynomials $\{\mathfrak{Q}_{n}\}_{n>m}$ and $\{\mathfrak{P}_{n}\}_{n>m}$ reads as

\begin{theorem}\label{RelativeAsymptoticLH}
Let $\{\zeta_n\}_{n>m}$  be a sequence of complex numbers  with limit  $\zeta \in \CC \setminus \Delta_c$, $m \in \NN$, $\mu \in \mathcal{P}_m(\Delta)$ and  $\{\mathfrak{Q}_n\}_{n>m}$ be the sequence of normalized monic orthogonal polynomials with respect to the pair $(\LL,\mu)$ such that $\mathfrak{Q}_n(\zeta_n)=0$, then:
\begin{enumerate}
  \item Uniformly on compact subsets  of  $\Omega=\{z \in \CC : |\Psi(z)| > |\Psi(\zeta)| \}$,
\begin{equation}
\label{AsintComp1}
\frac{\mathfrak{Q}_{n}(z)}{\mathfrak{P}_n(z)}   \unifn   1.
\end{equation}

\item Uniformly on compact subsets  of  $\Omega=\{z \in \CC : |\Psi(z)| < |\Psi(\zeta)| \} \setminus \Delta_c$
            \begin{equation}
\label{AsintComp2}
\frac{\mathfrak{Q}_{n}(z)}{\mathfrak{P}_n(\zeta_n)}   \unifn   -1,
\end{equation}
where $\Psi$ is as defined in Theorem \ref{zeroloc}.
If $\dsty \mathfrak{d}(\zeta)>2$ then \eqref{AsintComp2} holds for $\Omega=\{z \in \CC : |\Psi(z)| < |\Psi(\zeta)| \}$.
\end{enumerate}
\end{theorem}

The paper continues as follows.  Section \ref{Sec_EU}  is dedicated to the study of existence, uniqueness and  some results concerning the properties of the zeros of orthogonal polynomials with respect to the  Laguerre or Hermite operators. In Sections \ref{ZeroLocAsymtpBeh} and \ref{withouthat} we study the asymptotic behavior of the polynomials $\widehat{\mathfrak{Q}}_{n}$ and  $\mathfrak{Q}_{n}$ respectively.  Finally, in Section \ref{Sec_FluidLH} we show a fluid dynamics  model  for  the zeros of these polynomials.

\section{The polynomial ${Q}_{n}$}
\label{Sec_EU}

 First of all, we are  interested in discussing  systems of polynomials such that for some $m\in \NN$, for all $n>m$, they are solutions of \eqref{OrthDiff_02}. In order to classify those measures $\mu$ for which the existence of such sequences of orthogonal polynomials with respect to $(\LL,\mu)$ can be guaranteed,  we prove a preliminary lemma.

\begin{lemma} \label{Theo_EUverII} Let  $\mu$ be a finite positive Borel measure with support contained on $\RR$ and let $n\in \NN$ be fixed. Then, the differential equation \eqref{OrthDiff_02} has  a
monic polynomial solution $Q_n$ of degree $n$, which is unique up to an additive constant, if and only if
\begin{equation}\label{(11)verII}
\int P_n(x)dw(x)=0, \; \mbox{where $P_n$ is as  \eqref{OrthPoly_mu}.}
\end{equation}
\end{lemma}

\begin{proof} Suppose that there exists a polynomial $Q_n$ of degree $n$, such that $\LL[Q_n]=-n  \,P_n$. Then,   integrating \eqref{Lag_DO} or \eqref{Her_DO} with respect to the Laguerre measure on $\RR_+$ or Hermite measure on $\RR$ respectively we have \eqref{(11)verII}.

Conversely, suppose that  $P_n$ satisfies  \eqref{(11)verII}. Let $Q_n$ be the polynomial of degree $n$ defined by $\dsty  Q_n(z) = L_n(z) + \sum_{k=0}^{n-1}a_{n,k} L_k(z),$ where $a_{n,0}$ is an arbitrary constant and
$\dsty a_{n,k}= \frac{\lambda_n}{\lambda_{k}\tau_{k}}\int P_n(x) L_{k}(x) dw(x),$  $k=1,\ldots ,n-1.$ From the linearity of $\LL[\cdot]$ and \eqref{eigen_poly} we get that $\LL[Q_n]=-n \,P_n$.\end{proof}

From the preceding lemma,  as in \cite[Coroll. 2.2]{BorPij12},  we obtain

\begin{proposition} \label{(Cor1)LH} Let $w$  be the Laguerre or Hermite measure and $\mu$  a finite positive Borel measure on $\Delta$,   such that  $d\mu(x)=r(x) d w(x)$ with $r \in {L}^2(w)$. Then, $m$ is the smallest natural number such that for each $n>m$ there exists a monic polynomial $Q_n$ of degree $n$, unique up to an additive constant and orthogonal with respect to $(\LL, \mu)$ if and only if    $r^{-1}$ is a polynomial of degree   $m$.
\end{proposition}

\begin{proof} Suppose that $m$ is the smallest natural number such that  for each $n>m$ there exists a monic polynomial $Q_n$ of degree $n$, unique up to an additive constant and orthogonal with respect to $(\LL, \mu)$. According to  Lemma \ref{Theo_EUverII}
$$\int   L_n(x)\frac{d\mu(x)}{r(x)}=\int L_n(x)dw(x)
\left\{ \begin{array}{cc}
 =0 & \mbox{if } n>m,\\
 \neq 0 & \mbox{if } n=m.
  \end{array}
\right.$$
But this is equivalent to saying that $\dsty
\frac{1}{r(x)}= \sum_{k=0}^{m} c_{k} L_k(x) $ with $c_m \neq 0$. The converse is straightforward.
\end{proof}

It is possible to give another characterization, in terms of the quasi orthogonality concept,   for the existence of a system of polynomials such that for all $n>m$, for some $m\in \NN$, they are  solutions of \eqref{OrthDiff_02}.

\begin{proposition}\label{Theo_EU}
Let $\mu$ be a finite positive Borel measure on $\RR$ and $\{P_n\}_{n=0}^{\infty}$ the sequence of monic orthogonal polynomials with respect to $\mu$. Then, $m$ is the smallest natural number such that for each $n>m$ there  exists, except for an additive constant, a unique  monic polynomial  $Q_n$, orthogonal with respect to  the pair $(\LL,\mu)$, if and only if for all $n > m$
$$
\int P_n(x)\, x^k dw(x)=0, \quad \mbox{for } k=0,1, \dots, n-m,
$$
i.e. the polynomial $P_n$ is  quasi-orthogonal  of  order $n-m+1$ with respect to  the measure $w$.
\end{proposition}

\begin{proof}
Assume that $m$ is the smallest natural number such that  for each $n>m$ there exists a monic polynomial $Q_n$ of degree $n$, unique up to an additive constant and orthogonal with respect to $(\LL, \mu)$. From   Lemma \ref{Theo_EUverII}     we have that  \eqref{(11)verII} holds for $n>m$. From  the three term recurrence relation for $\{P_n\}_{n=0}^{\infty}$
\begin{eqnarray}\label{P_n_3trr}
  xP_n(x) & = & P_{n+1}(x)+ \beta_n P_n(x)+\alpha_n^2  P_{n-1}(x),\; n \geq 1,\\
 \nonumber  & & P_0(x)=1, \;  P_{-1}(x)=0, \; \alpha_n,\beta_n \in \RR \mbox{ and } \alpha_n \neq 0,\\
 \mbox{thus }\; \int P_n(x) x^k dw(x) &=& 0 \quad \mbox{for all } \;  0 \leq k<n -m,\label{EU_Cond_3}
\end{eqnarray}
which implies that the polynomial $P_n$ is quasi-orthogonal  of  order $n-m+1$ with respect to  the measure $w$ (Laguerre or Hermite).

Conversely, assume that $m$ is the smallest natural number such that for $n > m$, the polynomial $P_n$ is  quasi-orthogonal of  order $n-m+1$ with respect to  the measure $dw$. Then we have that
$$P_n(x)=L_n(x)+\sum_{k=1}^{m}d_{n-k}L_{n-k}(x),$$
which implies that for all integers $n > m$ the polynomials $P_n$ satisfy the condition  \eqref{(11)verII}. From Lemma \ref{Theo_EUverII} we have that  there exists a monic polynomial $Q_n$ of degree $n$, unique up to an additive constant and orthogonal with respect to $(\LL, \mu)$, for all $n>m$.
 \end{proof}

From the above proposition, we deduce in particular  that the differential equation \eqref{OrthDiff_02} has, except  for an additive constant, a unique monic polynomial solution $Q_n$ of degree $n$ \emph{for all the natural numbers} only if   $P_n=L_n$ and $d\mu=dw$. Hence  $Q_n=L_n$, the polynomial eigenfunctions of $\LL$, whose properties are well known.

Let us  continue by noting that  the polynomials $Q_n$ and $\dsty \widehat{Q}_n$ (see \eqref{Qhat_Def} and \eqref{(13)LH}) are primitives of the same polynomial $Q_n^{\prime}$ (or $\widehat{Q}_n^{\prime}$) and
\begin{equation}\label{SeudoOrtogo_04}
 \int  \, \widehat{Q}_n(x) \, x^k \, dw(x)=0, \quad k=0,1,\ldots,n-m-1.
\end{equation}

Applying classical arguments \cite{Sho37}, it is not difficult to prove the following result, which will be used in the sequel.

\begin{proposition}\label{Zero_Loc01} The polynomial $ \widehat{Q}_n$ defined by \eqref{Qhat_Def} for all $n>m$, has at least $(n-m)$ zeros and $(n-m-1)$ critical points of  odd multiplicity  on $\Delta$.
\end{proposition}

For  $m=2$ we denote by $\mathcal{\widetilde{P}}_2[\RR]$ the class of measures of the form $\dsty d\mu=\frac{e^{-x^2}}{x^2+x_1^2}dx,x_1\neq 0$ in the Hermite case. The following proposition shows some results concerning the zeros of $\widehat{Q}_n$ and $\widehat{Q}^{\prime}_n$ for measures $\mu \in \mathcal{P}_1[\RR_+]$ or $ \mu \in \mathcal{\widetilde{P}}_2[\RR]$.

\begin{proposition} \label{Zero_Loc02}
Assume that $\mu \in \mathcal{P}_1[\RR_+]$ or $ \mu \in \mathcal{\widetilde{P}}_2[\RR]$, then the zeros of $\widehat{Q}_n$ and $\widehat{Q}^{\prime}_n$ are real and simple. The critical points of $Q_n$   interlace  the zeros of $P_n$.
\end{proposition}

\begin{proof}\

\noindent 1.- \emph{Laguerre case.}
If $m=1$ and $\mu \in \mathcal{P}_1[\RR_+]$  from Proposition \ref{Zero_Loc01} the polynomial $\widehat{Q}_n$ has at least $(n-1)$ real zeros of odd multiplicity on  $\RR_+$. But, $\widehat{Q}_n$ is a polynomial with real coefficients and degree $n$, consequently the zeros of $\widehat{Q}_n$ are real and simple. As ${Q}^{'}_n  = \widehat{Q}^{'}_n $, from Rolle's theorem all the critical points of ${Q}_n$ are real, simple, and $(n-2)$ of them are contained on $\RR_+^*=]0,\infty[$.

 Denote  $\dsty G(x)= x^{\alpha+1}\,e^{-x}\, Q_n^{\prime}(x)$, with $\alpha \in ]-1,\infty[$. Notice that $G$ is  a real--valued, continuous and differentiable  function on $\RR_+^*$.  Suppose that there exists $x \in \RR_+^*$ such that $G(x)=0$. As $G(0)=0$ from Rolle's Theorem there exists $x^{\prime} \in \RR_+^*$ such that $G^{\prime}(x^{\prime})=0$. But, $\dsty G^{\prime}(x)=x^{\alpha}\,e^{-x} \LL_L[Q_n]=\lambda_n x^{\alpha}\,e^{-x}\, P_n(x)$ and all the critical points of $G$ are contained on $\RR_+^*$. Hence all the critical points of ${Q}_n$ belong to $\RR_+^*$.

\medskip\noindent 2.- \emph{Hermite case.}
Consider now $ \mu \in \mathcal{\widetilde{P}}_2[\RR]$, that is, $m=2$ and $\dsty d\mu(x)=\frac{e^{-x^2}}{x^2+x_1^2}dx$, $x_1\neq 0$. Using the relations \eqref{Qhat_Def} and \cite[5.6.1]{Szg75} we have that for $k>1$
\begin{eqnarray}
\label{HtL}\widehat{Q}_{2k}(z)&=&L^{-1/2}_{k}(z^2)+\frac{k}{k-1}\frac{\langle P_{2k},H_{2k-2},\rangle_H}{\langle H_{2k-2},H_{2k-2},\rangle_H} L^{-1/2}_{k-1}(z^2),\\
\nonumber \widehat{Q}_{2k+1}(z)&=&zL^{1/2}_{k}(z^2)+\frac{2k+1}{2k-1}\frac{\langle P_{2k+1},H_{2k-1},\rangle_H}{\langle H_{2k-1},H_{2k-1},\rangle_H}zL^{1/2}_{k-1}(z^2).
\end{eqnarray}

As $L^{-1/2}_{n}(z^2), zL^{1/2}_{n}(z^2)$ are the $2n$ and $2n+1$  monic orthogonal polynomials of degree $2n$ and $2n+1$  respectively with respect to the measure $d\mu(x)=e^{-x^2}dx$,  from   \eqref{HtL} and \cite[Th. 3.3.4]{Szg75} we have that the zeros of $\widehat{Q}_{n} , n>2$  are real.

   The statement that critical points of $Q_n$  interlace the zeros of $P_n$ follows by applying  Rolle's theorem to the functions $\dsty G(x)= x^{\alpha+1}\,e^{-x}\, Q_n^{\prime}(x)$ and $\dsty G(x)= e^{-x^{2}}\, Q_n^{\prime}(x)$,  for both the Laguerre and Hermite cases.
\end{proof}

We conjecture that Proposition \ref{Zero_Loc02} is still  valid  for any  measure in the class $\mathcal{P}_m(\Delta)$, $m>1$, for the Laguerre case  or $m>2$, $m$ even, for the Hermite case.

Finally, we find asymptotic bounds for the coefficients $b_{n,n-k}$ that define the polynomial $\widehat{Q}_{n}$.

\begin{proposition}\label{Lem_Coeff}
Let $m \in \NN$ and $\mu \in \mathcal{P}_m(\Delta)$. Then for $n$ large enough, there exist constants $C^{L}_{\rho}$ and  $C^{H}_{\rho}$ such that
$$
|b_{n,n-k}|= \frac{\left|\langle P_n, L_{n-k} \rangle_w \right|}{\|L_{n-k}\|_{w}^2}  <
\left\{ \begin{array}{cc}
C^{L}_{\rho}\,n^k & \text{Laguerre case,} \\ & \\
C^{H}_{\rho}\, \sqrt{n^{k}}  & \text{Hermite case,}
\end{array}
\right.
$$
for $k=1,\ldots,m$.
\end{proposition}

\begin{proof} Let  $\dsty\rho(x)=\sum_{j=1}^{m}\rho_jx^j$ and $\dsty \rho_+=\max_{0 \leq j \leq m}{|\rho_j|}$. From the Cauchy--Schwarz inequality we have
\begin{eqnarray}
 |b_{n,n-k}| & \leq  & \frac{\|P_n \|_{\mu}}{\|L_{n-k}\|_{w}^2} \sqrt{\langle \rho L_{n-k}, L_{n-k} \rangle_w} \leq    \frac{\|\rho L_{n-m} \|_{\mu}}{|\rho_m|\,\|L_{n-k}\|_{w}^2} \sqrt{\langle \rho L_{n-k}, L_{n-k} \rangle_w} \nonumber \\ \label{coeff0}
  & \leq  & \frac{\rho_+}{|\rho_m|\,\|L_{n-k}\|_{w}^2} \sqrt{\sum_{j=0}^{m}\left|\langle x^j, L_{n-m}^2 \rangle_w\right|} \sqrt{\sum_{j=0}^{m}\left|\langle x^j, L_{n-k}^2 \rangle_w\right|}
\end{eqnarray}

We analyze separately the  Laguerre and Hermite cases. Without loss of generality we can assume that $n>2m$.

\medskip\noindent\emph{$\bullet$ Laguerre case $(L_n=L_n^{\alpha},\; \Delta=\RR_{+} \;  \mbox{ and } \;dw(x)=x^{\alpha}e^{-x}dx)$.} From \cite[(III.4.9) and (I.2.9)]{Rus05} we have the connection formula
\begin{equation*}\label{Connect01}
L_{n-k}^{\alpha}(z)= \sum_{\nu=k}^{k+j} \binom{j}{\nu-k}\frac{(n-k)!}{(n-\nu)!} L_{n-\nu}^{\alpha+j}(z),
\end{equation*}
then  from  \eqref{Lag-Herm_Norm2} and the orthogonality
\begin{eqnarray*}
\langle x^j, \left( L_{n-k}^{\alpha}\right)^2 \rangle_L  &=&   \sum_{\nu=k}^{k+j} \binom{j}{\nu-k}\frac{(n-k)!}{(n-\nu)!}  \int \,\left(L_{n-\nu}^{\alpha+j}(x)\right)^2 x^{\alpha+j} e^{-x}dx, \\
   &=&   \sum_{\nu=k}^{k+j} \binom{j}{\nu-k} (n-k)! \Gamma(n-\nu+j+\alpha+1) ,\\
    & \leq & 2^j (n-k)! \Gamma(n-k+j+\alpha+1) ,\\ \mbox{and }\; 
\sum_{j=0}^{m} \langle x^j, L_{n-k}^2 \rangle_w & \leq &  (n-k)! \sum_{j=0}^{m} 2^j  \Gamma(n-k+j+\alpha+1), \\
   & \leq &  (2^{m+1}-1) (n-k)!  \Gamma(n-k+m+\alpha+1).
\end{eqnarray*}
Hence, from  \eqref{coeff0}, \eqref{Lag-Herm_Norm2} and $n$  large enough
\begin{eqnarray*}
\nonumber  |b_{n,n-k}|    & \leq  & \frac{\rho_+ (2^{m+1}-1)}{|\rho_m|} \sqrt{\frac{(n-m)! \Gamma(n+\alpha+1) \Gamma(n+m-k+\alpha+1)}{(n-k)! \Gamma^2(n-k+\alpha+1)}},
\\
& \leq & \frac{\rho_+ (2^{m+1}-1)}{|\rho_m|} \sqrt{\frac{(n+\alpha)^{k+m}}{(n-m)^{m-k}}}  \leq  \frac{\rho_+ 2^m(2^{m+1}-1)}{|\rho_m|}\; n^k.
\end{eqnarray*}

\medskip\noindent\emph{$\bullet$ Hermite case  $(L_n=H_n,\; \Delta=\RR \; \mbox{ and } \;dw(x)=e^{-x^2}dx)$. } By the symmetry property of the Hermite polynomials, if $\nu$ is an odd number
$$
\int x^{\nu}\,H_{n-k}^{2}(x)dw(x)=0.
$$
Hence, from \eqref{coeff0}
$$
|b_{n,n-k}|  \leq   \frac{\rho_+}{|\rho_m|\,\|H_{n-k}\|_{w}^2} \sqrt{\sum_{j=0}^{\left\lfloor\frac{m}{2}\right\rfloor}\|x^j\,H_{n-m}\|_{w}^{2}} \sqrt{\sum_{j=0}^{\left\lfloor\frac{m}{2}\right\rfloor} \|x^j\,H_{n-k}\|_{w}^{2}},
$$
where for all $x \in \RR$, the symbol $\left\lfloor x\right\rfloor$ denote the  largest integer less than or equal to $x$. As it is well known (cf. \cite[(5.5.6) and (5.5.8)]{Szg75}), the Hermite polynomials satisfy the recurrence relation
$z H_{n}(z)= H_{n+1}(z)+ \frac{n}{2} H_{n-1}(z),
$ from which we get by induction on $j$
\begin{equation}\label{RecuReiterate}
z^j H_{n}(z)= \sum_{\nu=0}^{j} \sigma_{j,\nu}(n) H_{n+j-2\nu}(z),
\end{equation}
where $\sigma_{j,\nu}(n)$ is a polynomial in $n$ of degree equal to $\nu$ and leading coefficient $2^{-\nu}\binom{j}{v}$ (i.e. $\sigma_{j,\nu}(n)= 2^{-\nu}\binom{j}{v} n^{\nu}+ \cdots $). Hence,  from \eqref{Lag-Herm_Norm2}, for $n$ large enough
\begin{eqnarray*}
  \|x^j H_{n-k}\|_{w}^2 &=&  \sum_{\nu=0}^{j} \sigma_{j,\nu}^2(n-k) \|H_{n-k+j-2\nu}\|_{w}^2, \\
      & \leq & \frac{\sqrt{\pi}\;(n-k-j)!}{2^{n-k+j}}\left(\sum_{\nu=0}^{j} 2^{2\nu} \sigma_{j,\nu}^2(n-k) (n-k+j)^{2j-2\nu} \right), \\ & \leq &  \frac{2 \sqrt{\pi}\;(n-k-j)! (n-k)^{2j}}{2^{n-k}}\binom{2j}{j},
\end{eqnarray*} with $j=0,1, \ldots, \left\lfloor\frac{m}{2}\right\rfloor$, therefore
\begin{eqnarray*}
  |b_{n,n-k}| & \leq  & \frac{\rho_+ \;2^{n-k}}{\sqrt{\pi}|\rho_m|\,(n-k)!} \sqrt{\sum_{j=0}^{\left\lfloor\frac{m}{2}\right\rfloor}\|x^j\,H_{n-m}\|_{w}^{2}} \sqrt{\sum_{j=0}^{\left\lfloor\frac{m}{2}\right\rfloor} \|x^j\,H_{n-k}\|_{w}^{2}}\\
   & \leq &  \frac{2 m! \rho_+ }{|\rho_m|} \sqrt{\sum_{j=0}^{\left\lfloor\frac{m}{2}\right\rfloor}(n-m)^{ 2j} \frac{\left(n-m-j\right)!}{(n-k)!}} \sqrt{\sum_{j=0}^{\left\lfloor\frac{m}{2}\right\rfloor} (n-k)^{ 2j} \frac{\left(n-k-j\right)!}{{(n-k)!}}} \\
     & \leq & \frac{2 m! \rho_+ }{|\rho_m|} \sqrt{\sum_{j=0}^{\left\lfloor\frac{m}{2}\right\rfloor} \frac{(n-m)^{ 2j}}{(n-m-j)^{m+j-k}}} \sqrt{\sum_{j=0}^{\left\lfloor\frac{m}{2}\right\rfloor} \frac{(n-k)^{ 2j}}{(n-m-j)^j}}  \\
          & \leq &  \frac{2 m! \rho_+ }{|\rho_m|} \sqrt{ 8 m (n-k)^{-\left\lfloor\frac{m}{2}\right\rfloor}} \sqrt{2 m (n-k)^{\left\lfloor\frac{m}{2}\right\rfloor}} \, n^k=
          \frac{8 m (m)! \rho_+ }{|\rho_m|} \;n^k.
\end{eqnarray*}\end{proof}

\section{The polynomial $\mathfrak{\widehat{Q}}_{n}$}\label{ZeroLocAsymtpBeh}

In this section we prove   asymptotic properties  of the normalized monic orthogonal polynomials with respect to a Laguerre or Hermite differential operator. We recall that as in Section \ref{SecIntro},  $\Delta_c$ denotes the interval  $[0,1]$ in the Laguerre case and $[-1,1]$ in the Hermite case, and the sequence of real numbers $\{c_n\}_{n=1}^{\infty}$ is given by \eqref{LargestZ-LH}.  Set  $\mathfrak{L}_{n,\nu}(z)= c_n^{-\nu}\;L_{\nu}(c_nz); \mathfrak{L}_{n}(z)\equiv\mathfrak{L}_{n,n}(z)$ and  $\mathfrak{P}_{n,\nu}(z)= c_n^{-\nu}\;P_{\nu}(c_nz); \mathfrak{P}_{n}(z)\equiv\mathfrak{P}_{n,n}(z)$.

We prove now some preliminary lemmas.

\begin{lemma}\label{loc}
Let $m \in \NN$, $\mu \in \mathcal{P}_m(\Delta)$ and $\zeta$ such that $\widehat{\mathfrak{Q}}_n(\zeta)=0$. Then for all $n$ sufficiently large
$d_c(\zeta)< 2\varpi_c,$ where

$$\varpi_c=\left\{ \begin{array}{ll}
   1+ 2^{-1}\,C^{L}_{\rho} & \mbox{Laguerre case},\\
   1+ \sqrt{2}\,C^{H}_{\rho} & \mbox{Hermite case},
      \end{array}
 \right.$$
 $\dsty d_c(z)= \min_{x \in \Delta_c}|z-x|$, and $C^{L}_{\rho}$ and  $C^{H}_{\rho}$ are the same constants of Proposition \ref{Lem_Coeff}.
\end{lemma}

\begin{proof}
For each fixed $n>m$,  we have that
$$  x_n^{-n} \widehat{Q}_n(x_nz)= \sum_{k=0}^{m} \frac{ \lambda_n\,b_{n,n-k}}{x_n^k\lambda_{n-k}} \;  x_n^{-n+k}L_{n-k}(x_nz),$$
where $x_n$ is the zero of the largest modulus of $L_n$. It follows that the smallest interval containing the zeros of  $\{x_n^{-k}L_{k}(x_nz)\}_{k=0}^{n}$ is $\Delta_{c}$. Hence, if $\zeta$ is such that $\widehat{Q}_n(x_n\zeta)=0$, from \cite[Coroll.  1]{Sch98},  Proposition \ref{Lem_Coeff},  \eqref{LargestZ-LH}, and \eqref{asympzero} we have,
\begin{eqnarray}\label{preloc}
d_c(\zeta)  &\leq & 1+\max_{1\leq k \leq m}\left|\frac{ \lambda_n\,b_{n,n-k}}{x_n^k\lambda_{n-k}}\right|<1+2\max_{1\leq k \leq m}\left|\frac{ b_{n,n-k}}{x_n^k}\right|\leq \varpi_c,
\end{eqnarray}
where
$$
\varpi_c=\left\{ \begin{array}{ll}
   1+ 2^{-1}\,C^{L}_{\rho} & \mbox{Laguerre case},\\
   1+ \sqrt{2}\,C^{H}_{\rho} & \mbox{Hermite case}.
      \end{array}
 \right.
$$
Notice that $\dsty \mathfrak{\widehat{Q}}_n\left(\frac{x_n}{c_n}z\right)= c_{n}^{-n}\widehat{Q}_n(x_nz)$; therefore, if   $\zeta$ is such that $\widehat{Q}_n(x_n\zeta)=0$ then $\dsty\zeta^*=\frac{x_n}{c_n}\zeta$ is such that $\mathfrak{\widehat{Q}}_n(\zeta^*)=0$. From \eqref{asympzero} and \eqref{LargestZ-LH} we have that for $n$ large, $\dsty\left| \frac{x_n}{c_n}\right|<2$. Using now \eqref{preloc} we obtain the lemma.
\end{proof}

If  $\{\Pi_n\}_{n=0}^{\infty}$ is a sequence of orthogonal polynomials with respect to either  the measures $\mu$ or $w$ we denote by $\{\mathfrak{\Pi}_n\}_{n=0}^{\infty}$  the sequence of monic normalized polynomials, that is,
\begin{equation}\label{genricPL}
\mathfrak{\Pi}_n(z)=c_n^{-n}\Pi_n(c_n z) \; \mbox{ and }\; \mathfrak{\Pi}_{n,\nu}(z)=c_n^{-\nu}\Pi_{\nu}(c_n z).
\end{equation}
From the interlacing property of the zeros of consecutive orthogonal polynomials, if $K$ is a compact subset of $\CC\setminus \Delta_c$  it  follows that there exist a   constant $ M_*$  such that for $n$ large enough

\begin{equation}
 \label{Lacotger} \left|\frac{ \mathfrak{\Pi}_{n,n-k}(z)}{\mathfrak{\Pi}_{n}(z)} \right| < M_k\leq  M_*, \quad k=1,\ldots ,m,
\end{equation}
 uniformly on  $z \in K$, where $\dsty M_k=2\sup_{\substack{ z \in K \\ x \in \Delta_c}} |z-x|^{-k}$, $M_* = \max\{M_1,\ldots, M_m\}$.

The following lemma is needed to study the modulus of the sequence $\dsty \left\{ \frac{\mathfrak{P}_{n} }{\mathfrak{L}_{n} }\right\}_{n=0}^{\infty}$.

\begin{lemma}\label{h}
Suppose that $m \in \NN$ is fixed, and  $K\subset \mathbb{C}\setminus \Delta_c$ a compact subset. Then, for $n$ sufficiently large
\begin{eqnarray}
\label{h0} \dsty \left |\left(\frac{c_{n+m}}{c_n}\right)^{n} \frac{\mathfrak{\Pi}_n(z)}{\mathfrak{\Pi}_n(\frac{c_{n+m}}{c_n}z)}  \right |&< &3^{\dsty\frac{2m}{d}},\quad n>n_0, \forall z\in K,
\end{eqnarray}
where $\dsty  d=\inf_{\substack{ z \in K \\ x \in \Delta_c}}|z-x|$ and $\mathfrak{\Pi}_n$ as in \eqref{genricPL}.
\end{lemma}

\begin{proof}Let us define the monic polynomial $\dsty \mathfrak{\Pi}^*_{n}(z)=\left(\frac{c_n}{c_{n+m}}\right)^{n}\mathfrak{\Pi}_n\left(\frac{c_{n+m}}{c_n}z\right)$. We have  that \eqref{h0} is equivalent to proving that
$$\dsty \left |\frac{\mathfrak{\Pi}_n(z)}{\mathfrak{\Pi}^*_{n}(z)}  \right |\leq 3^{\dsty\frac{2m}{d}},\quad n>n_0, \forall z\in K.$$
If $\{z^{*}_{k,n}\}_{k=1}^{n}$, $\{z_{k,n}\}_{k=1}^{n}$ denotes the  zeros of the polynomials $\mathfrak{\Pi}^*_{n}$, $\mathfrak{\Pi}_{n}$ respectively, we have the relation  $\dsty z^{*}_{k,n}=\frac{c_n}{c_{n+m}}z_{k,n}, k=1,\ldots ,n$. If we denote $\dsty k_n=\frac{c_n}{c_{n+m}}$, we have, for all $n$ sufficiently large
\begin{eqnarray}
\dsty \left |\frac{\mathfrak{\Pi}_n(z)}{\mathfrak{\Pi}^*_{n}(z)} \right |&\leq & \left |\prod_{k=1}^{n}\left(1+\frac{(k_n-1)z_{k,n}}{z-k_nz_{k,n}}\right)  \right |\leq \prod_{k=1}^{n} \left(1+|k_n-1|\left|\frac{z_{k,n}}{z-k_nz_{k,n}}\right|\right) \\
\nonumber &\leq &\prod_{k=1}^{n} \left(1+ \frac{2|k_n-1|}{d} \right)\leq \left(1+ \frac{2|k_n-1|}{d} \right)^{n}< 3^{\dsty\frac{2n|k_n-1| }{d}}\leq 3^{\dsty\frac{2m}{d}},
\end{eqnarray}
where  $\dsty  d=\inf_{\substack{ z \in K \\ x \in \Delta_c}}|z-x|$.
\end{proof}

We prove now that the modulus of  the sequence $\dsty \left\{ \frac{\mathfrak{P}_{n} }{\mathfrak{L}_{n} }\right\}_{n=0}^{\infty}$ is uniformly bounded from above and below in the interior of $\mathbb{C}\setminus \Delta_c$.

\begin{lemma}\label{LPright}
Let  $\mu \in \mathcal{P}_m(\Delta)$, where $m \in \NN$ and  $K\subset \mathbb{C}\setminus \Delta_c$ a compact subset. Then, for all $n$ sufficiently large there exists a constant $C^*$ such that
$$\dsty \left | \frac{\mathfrak{P}_{n}(z)}{\mathfrak{L}_{n}(z)}\right |\leq C^*,\quad n>n_0, \forall z\in K.$$
\end{lemma}

\begin{proof}
From Relation \eqref{PnLn} we deduce that
$\dsty
\frac{\mathfrak{P}_n(z)}{\mathfrak{L}_n(z)}=1 + \sum_{k=1}^{m}\frac{b_{n,n-k}}{c_n^k}  \frac{\mathfrak{L}_{n,n-k}(z)}{\mathfrak{L}_n(z)}.
$ Hence, from Proposition \ref{Lem_Coeff}, and Lemma \ref{h} we deduce that for $n$ large enough
\begin{eqnarray}\label{RP}
\left|\frac{\mathfrak{P}_n(z)}{\mathfrak{L}_n(z)}\right|\leq  1 +  \sum_{k=1}^{m} C_{\rho}\left|\frac{\mathfrak{L}_{n,n-k}(z)}{\mathfrak{L}_n(z)}\right|,
\end{eqnarray}
Using  \eqref{RP} and \eqref{Lacotger} we deduce the lemma.
\end{proof}

\begin{lemma}\label{LPleft}
Let  $\mu \in \mathcal{P}_m(\Delta)$, where $m \in \NN$ and  $K\subset \mathbb{C}\setminus \Delta_c$ is a compact subset. Then, for all $n$ sufficiently large there exists a constant $C$ such that
$$\dsty C\leq \left | \frac{ \mathfrak{P}_{n}(z)}{\mathfrak{L}_{n}(z)}\right |,\quad n>n_0, \forall z\in K.$$
\end{lemma}

\begin{proof} We have that $\dsty  \rho(z)L_n( z) =\sum_{k=0}^{m} \mathfrak{b}_{n,n-k} P_{n+m-k}(z),$ where $\dsty
 \mathfrak{b}_{n,n-k}=\frac{\int L_{n}(x)P_{n+m-k}(x)\rho(x) d\mu(x)}{\|P_{n+m-k}(x)\|_{\mu}^2},$ or equivalently,
\begin{eqnarray}\label{LgPg}
\frac{\rho(c_{n+m}z)}{c^m_{n+m}}\left(\frac{c_n}{c_{n+m}}\right)^{n}\frac{\mathfrak{L}_n(\frac{c_{n+m}}{c_n} z)}{\mathfrak{L}_n(z)}\frac{ \mathfrak{L}_{n}(z)}{\mathfrak{P}_{n+m}(z)}=\sum_{k=0}^{m}\frac{\mathfrak{b}_{n,n-k}}{c_{n+m}^{k}}\frac{\mathfrak{P}_{n+m,n+m-k}(z)}{\mathfrak{P}_{n+m}(z)}.
\end{eqnarray}
From the Cauchy Schwartz inequality we have that

$$|\mathfrak{b}_{n,n-k}|\leq \frac{\left(\int L_n^2(x)dw(x)\right)^{1/2}(\int P_{n+m-k}^2(x)dw(x))^{1/2}}{\|P_{n+m-k}\|_{\mu}^2}=\frac{\|L_{n}\|_{w}\|P_{n+m-k}\|_{w}}{\|P_{n+m-k}\|_{\mu}^2}.$$

Using an infinite--finite range  inequality  for the case in which $w$ is a Laguerre weight, cf. \cite{safftotik97}, we have that  there exists a constant $k_L$ such that for all $n$ large enough
$$
\frac{k_L}{n^m}\int_{0}^{\infty} L^2_n(x) dw(x)\leq  \frac{k_{0,L}}{(4n)^m}\int_{0}^{\infty} P^2_n(x) dw(x) \leq  \frac{1}{\rho_+(4n)^m}\int_{0}^{4n} P^2_n(x) dw(x) \leq   \int_{0}^{\infty} P_{n}^2(x)d\mu(x),
$$
where $\dsty \rho_+=\max_{0 \leq j \leq m}{|\rho_j|}$. Analogously,  for the case of an Hermite weight, for all $n$ large enough, we have that there exists a constant $k_H$ such that
$$
\frac{k_H}{n^{m/2}}\int_{-\infty}^{\infty} L^2_n(x) dw(x) \leq  \frac{k_{0,H}}{(2n)^{m/2}}\int_{-\infty}^{\infty} P^2_n(x) dw(x) \leq  \frac{1}{\rho_+(2n)^{m/2}}\int_{-\sqrt{2n}}^{\sqrt{2n}} P^2_n(x) dw(x)\leq\int_{-\infty}^{\infty} P_{n}^2(x)d\mu(x),
$$

Hence, for all $n$ large enough
\begin{eqnarray}\label{PnmLnw}
\dsty \|P_n\|^2_{\mu} &\geq & k_Ln^{-m}\|L_{n}\|^2_{w},\quad \mbox{Laguerre case,}\\
\nonumber \dsty \|P_n\|^2_{\mu} &\geq & k_Hn^{-m/2}\|L_{n}\|^2_{w},\quad \mbox{Hermite case.}
\end{eqnarray}

From    \eqref{PnLn} and Proposition \ref{Lem_Coeff} we deduce that for $n$ large enough, there exists a constant $k_1$ such that
\begin{equation}\label{PnwLnw}
\|P_n\|_{w}\leq k_1\|L_n\|_{w}.
\end{equation}
Inequalities \eqref{PnmLnw} and \eqref{PnwLnw} give us that there exists a constant $M^*$ such that for all $n$ large enough
\begin{equation}\label{bgerm}
\frac{|\mathfrak{b}_{n,n-k}|}{c_{n+m}^k}\leq  M^*, \quad 1\leq k\leq m.
\end{equation}
From \eqref{Lacotger} it  follows that there exists  a   constant $ M_*$  such that for all $z \in K$
\begin{equation}\label{Pacot}
 \left|\frac{ \mathfrak{P}_{n+m,n+m-k}(z)}{\mathfrak{P}_{n+m}(z)} \right| <  M_*, \quad k=1,\ldots ,m.
\end{equation}
Using  Lemma \ref{h}, \eqref{LgPg}, \eqref{bgerm} and \eqref{Pacot} we obtain
\begin{equation}\label{PreLgPgnormal}
\dsty \left |\frac{\rho(c_{n+m}z)}{c^m_{n+m}}\right |\left | \frac{  \mathfrak{L}_{n}(z)}{\mathfrak{P}_{n+m}(z)}\right |\leq 3^{\dsty\frac{2m}{d}}\left(1+m\,M^{*}\,M_*\right),
\end{equation}
with $d$ as in Lemma \ref{h}. Hence, from  \eqref{Lacotger}, \eqref{Pacot}, \eqref{PreLgPgnormal} and Lemma \ref{h} we obtain that for all $n$ sufficiently large there exists  $M>0$  such that
\begin{equation}\label{acotuniff}
\dsty \left |\frac{\rho(c_{n+m}z)}{c^m_{n+m}}\right |\left | \frac{\mathfrak{L}_{n}(z)}{\mathfrak{P}_{n}(z)}\right |\leq M, \quad \forall z \in K.
\end{equation}

Let us denote by  $\{z_{k}\}_{k=1}^{m}$  the roots of the polynomial   $\rho$, and  $\dsty d^*=\inf_{z\in K}|z|$. Let us choose  $\varepsilon$ so that for $n$ large enough $\dsty \left|\frac{z_{k}}{c_{n+m}}\right|<\varepsilon<d^*, k=1,\ldots ,m$. Hence,
\begin{equation}\label{acotuniff1}
\left(d^*-\varepsilon\right)^m  \leq \prod_{k=1}^{m}\left(\left|z\right |-\left|\frac{z_k}{c_{n+m}}\right|\right)\leq \prod_{k=1}^{m}\left |\left(z-\frac{z_k}{c_{n+m}}\right)\right |= \left |\frac{\rho(c_{n+m}z)}{c^m_{n+m}}\right |.
\end{equation}

Therefore, from \eqref{acotuniff} and \eqref{acotuniff1}, for all $n$ large enough we have that
\begin{equation*}
\dsty  \left | \frac{\mathfrak{L}_{n}(z)}{\mathfrak{P}_{n}(z)}\right |\leq  \frac{M}{\left(d^*-\varepsilon\right)^m}, \quad \forall z \in K,
\end{equation*}
and this proves the lemma.
\end{proof}

\begin{lemma}\label{QLAcotac}
Let  $\mu \in \mathcal{P}_m(\Delta)$, where $m \in \NN$ and  $K\subset \mathbb{C}\setminus \Delta_c$ is a compact subset. Then,
$$\dsty\left | \frac{\mathfrak{\widehat{Q}}_n(z)}{\mathfrak{L}_{n}(z)}- \frac{\mathfrak{P}_n(z)}{\mathfrak{L}_{n}(z)}\right | \rightrightarrows 0, \quad \forall z\in K.$$
\end{lemma}

\begin{proof}

 For each fixed $n>m$,  we have that
\begin{equation}\label{QLacot0}
\frac{\widehat{\mathfrak{Q}}_n(z)-\mathfrak{P}_{n}(z)}{\mathfrak{L}_{n}(z)}= \sum_{k=0}^{m} \left(\frac{\lambda_n}{\lambda_{n-k}}-1\right)\frac{b_{n,n-k}}{c_n^k} \;  \frac{\mathfrak{L}_{n,n-k}(z)}{\mathfrak{L}_{n}(z)}.
\end{equation}

As $\lambda_n=-n$ in the Laguerre case and $\lambda_n=-2n$ in the Hermite case, then for each $k$ fixed, $k=1,\ldots,m$,
\begin{equation}\label{eigenvalueslim}
\lim_{n \to \infty}\frac{ \lambda_n}{\lambda_{n-k}}= 1.
\end{equation}
 From \eqref{Lacotger}, \eqref{QLacot0}, \eqref{eigenvalueslim} and Proposition \ref{Lem_Coeff} we deduce the lemma.
\end{proof}

\begin{proof}\emph{[Theorem \ref{countingmeasureconv}]}

$a)$ From \cite[(5.1.14),(5.5.10)]{Szg75} we have that
$\dsty \mathfrak{L}^{\prime}_{n,n-k}=(n-k)\mathfrak{\widetilde{L}}_{n,n-1-k}$, where

$$ \mathfrak{\widetilde{L}}_{n,n-1-k}=\begin{cases}
\dsty  c^{-(n-1-k)}_nL_{n-1-k}^{\alpha+1}(c_nz),& \quad \text{Laguerre case},\\
\dsty   c^{-(n-1-k)}_nH_{ n-1-k}(c_nz), &  \quad \text{Hermite case}.\\
\end{cases}$$

Let us define
\begin{eqnarray*}
d\widetilde{w}(x)&=&\begin{cases}
\dsty dw_{L}^{\alpha+1}(x),&  \text{Laguerre case},\\
\dsty dw_{H}(x), &  \text{Hermite case}.\\
\end{cases}\\
dw_n(x)&=&\begin{cases}
\dsty c_n^{-1}dw_{L}^{\alpha}(c_nx),&  \text{Laguerre case},\\
\dsty c_n^{-1}dw_{H}(c_nx), &  \text{Hermite case}.\\
\end{cases}\\
d\widetilde{w}_n(x)&=&\begin{cases}
\dsty c_n^{-1}dw_{L}^{\alpha+1}(c_nx),&  \text{Laguerre case},\\
\dsty c_n^{-1}dw_{H}(c_nx), &  \text{Hermite case}.\\
\end{cases}
\end{eqnarray*}

Notice that $\{\mathfrak{L}_{n,n-k}\}_{k=0}^{n}$ and $\{\mathfrak{\widetilde{L}}_{n,n-k}\}_{k=0}^{n}$  are monic orthogonal polynomials with respect to $w_n, \widetilde{w}_n$ respectively, hence, from \cite[(11)]{GoRak86}, we have that the sequences $\{\mathfrak{L}_{n,n-k}\}_{n=0}^{\infty}$ and $\{\mathfrak{\widetilde{L}}_{n,n-k}\}_{n=0}^{\infty}$ for every $k=0,\ldots ,m$ satisfy that

\begin{eqnarray}\label{GoRak}
\lim_{n\to\infty}\|w_n\mathfrak{L}_{n,n-k}\|_{L^{2}(\Delta)}^{1/n}=e^{-F_w}, \lim_{n\to\infty}\|\widetilde{w}_n\mathfrak{\widetilde{L}}_{n,n-k}\|_{L^{2}(\Delta)}^{1/n}=e^{-F_w},
\end{eqnarray}
where $F_w$ is the modified Robin constant for the weights $w,\widetilde{w}$ (or the extremal constant according to the terminology of \cite{GoRak86})  and $\|.\|_{L^{2}(\Delta)}$ denotes the $L^2$--norm with the Lebesgue measure with support on $\Delta$.

From \cite[Ths. 1 \& 2]{markett80} we have that
\begin{eqnarray}\label{nik}
\|w_n\mathfrak{L}_{n,n-k}\|_{L^{\infty}(\Delta)} &\leq & K_{1}n^{\beta}\|w_n \mathfrak{L}_{n,n-k}\|_{L^{2}(\Delta)},\\
\nonumber \|\widetilde{w}_n \mathfrak{\widetilde{L}}_{n,n-k}\|_{L^{\infty}(\Delta)} &\leq & K_{2}n^{\beta}\|\widetilde{w}_n \mathfrak{\widetilde{L}}_{n,n-k}\|_{L^{2}(\Delta)},
\end{eqnarray}
where  $K_{1}, K_{2}$ are constants that do not depend on $n$, $\beta=1/2$ for the Laguerre case, and $\beta=1/4$ for the Hermite case. Using \eqref{GoRak}, \eqref{nik}, and \cite[(11)]{GoRak86} we obtain that

\begin{eqnarray}\label{GoRakinfty}
\lim_{n\to\infty}\|w_n\mathfrak{L}_{n,n-k}\|_{L^{\infty}(\Delta)}^{1/n}=e^{-F_w}, \lim_{n\to\infty}\|\widetilde{w}_n\mathfrak{\widetilde{L}}_{n,n-k}\|_{L^{\infty}(\Delta)}^{1/n}=e^{-F_w}.
\end{eqnarray}

Then we have
\begin{eqnarray*}
\|w_n{\widehat{\mathfrak{Q}}}_n\|_{L^{\infty}(\Delta)}  &\leq &  \sum_{k=0}^{m} \left|\frac{ \lambda_n\,b_{n,n-k}}{c_n^k\lambda_{n-k}}\right|\; \|w_n \mathfrak{L}_{n,n-k}\|_{L^{\infty}(\Delta)} \\
&\leq & \left|\frac{(m+1)\lambda_n\,b_{n,n-k^*(n)}}{c_n^{k^*(n)}\lambda_{n-k^*(n)}}\right|\; \|w_n\mathfrak{L}_{n,n-k^*(n)}\|_{L^{\infty}(\Delta)}  ,
\end{eqnarray*}
and
\begin{eqnarray*}
 \|\widetilde{w}_n{\widehat{\mathfrak{Q}}}^{\prime}_n\|_{L^{\infty}(\Delta)} &\leq &   \dsty \sum_{k=0}^{m} \left|\frac{(n-k) \lambda_n\,b_{n,n-k}}{c_n^{k}\lambda_{n-k}}\right|\; \|\widetilde{w}_n\mathfrak{\widetilde{L}}_{n,n-1-k}\|_{L^{\infty}(\Delta)}\\
& \leq&  \left|\frac{(m+1)(n-k^{**}(n)) \lambda_n\,b_{n,n-k^{**}(n)}}{c_n^{k^{**}(n)}\lambda_{n-k^{**}(n)}}\right|\; \|\widetilde{w}_n\mathfrak{\widetilde{L}}_{n,n-1-k^{**}(n)}\|_{L^{\infty}(\Delta)},
\end{eqnarray*}
where $\|.\|_{L^{\infty}(\Delta)}$ denotes the sup norm and $k^*(n), k^{**}(n)$ denote positive integer numbers such that the following equalities hold
$$\left|\frac{ \lambda_n\,b_{n,n-k^*(n)}}{c_n^{k^*(n)}\lambda_{n-k^*(n)}}\right|\|w_n \mathfrak{L}_{n,n-k}\|_{L^{\infty}(\Delta)}=\max_{k=0,\ldots,m}\left|\frac{ \lambda_n\,b_{n,n-k}}{c_n^k\lambda_{n-k}}\right|\; \|w_n \mathfrak{L}_{n,n-k}\|_{L^{\infty}(\Delta)},$$
\begin{eqnarray*}
 & & \left|\frac{(n-k^{**}(n)) \lambda_n\,b_{n,n-k^{**}(n)}}{c_n^{k^{**}(n)}\lambda_{n-k^{**}(n)}}\right|\; \|\widetilde{w}_n\mathfrak{\widetilde{L}}_{n,n-1-k^{**}(n)}\|_{L^{\infty}(\Delta)} \\ & & = \max_{k=0,\ldots,m}\left|\frac{(n-k) \lambda_n\,b_{n,n-k}}{c_n^k\lambda_{n-k}}\right|\; \|\widetilde{w}_n \mathfrak{\widetilde{L}}_{n,n-k}\|_{L^{\infty}(\Delta)}.
\end{eqnarray*}

From these last inequalities and  \eqref{GoRakinfty} we deduce that
$$
\lim_{n\to \infty}\left(\|w_n{\widehat{\mathfrak{Q}}}_n\|_{L^{\infty}(\Delta)}\right)^{1/n}=e^{-F_w}, \;
\lim_{n\to \infty}\left(\|\widetilde{w}_n{\widehat{\mathfrak{Q}}}^{\prime}_n\|_{ L^{\infty}(\Delta)}\right)^{1/n}=e^{-F_w},
$$
Therefore, if $\nu_n,\delta_n$ denote the root counting measure of ${\widehat{\mathfrak{Q}}}_n$ and ${\widehat{\mathfrak{Q}}}^{\prime}_n$   respectively, from  \cite[Th. 1.1]{MeOrPi}  we deduce that $\nu_n\weak \nu_{w}$, $\delta_n\weak \nu_{w}$ in the weak star sense.

$b)$ From  Lemma \ref{LPleft}, if $\varepsilon$ is sufficiently small and $K\subset \CC\setminus\Delta_c$ is a compact subset, for all $n$ sufficiently large we have that, for some positive constant $C$,
\begin{equation*}
C-\varepsilon\leq \left|\frac{\mathfrak{P}_{n}(z)}{\mathfrak{L}_{n}(z)}\right|-\varepsilon\leq  \left|\frac{\widehat{\mathfrak{Q}}_n(z)}{\mathfrak{L}_{n}(z)}\right|.
\end{equation*}

 From this fact and from Lemma \ref{loc} we deduce that the set of accumulation points is contained on $\Delta_c$ and from $a)$ of the present theorem we deduce that the set of accumulation points of the zeros of ${\widehat{\mathfrak{Q}}}_n$ is $\Delta_c$.
\end{proof}

\begin{proof}\emph{[Theorem \ref{Th2}]}  From $b)$ of Theorem \ref{countingmeasureconv} we deduce that for the Laguerre case
\begin{eqnarray*}
\lim_{n\to\infty}\frac{\widehat{Q}_n^{\prime}(c_nz)}{\widehat{Q}_n(c_nz)}=\lim_{n\to\infty}\frac{\widehat{Q}_n^{\prime\prime}(c_nz)}{\widehat{Q}_n^{\prime}(c_nz)}=\frac{1}{2\pi }\int_{0}^{1}\frac{1}{z-t}\sqrt{\frac{1-t}{t}}dt =\frac{1}{2}\left(1-\sqrt{1-1/z}\right),
\end{eqnarray*}
and for the Hermite case
\begin{eqnarray*}
\lim_{n\to\infty} \frac{\widehat{Q}_n^{\prime}(c_nz)}{c_n\widehat{Q}_n(c_nz)}=\lim_{n\to\infty}\frac{\widehat{Q}_n^{\prime\prime}(c_nz)}{c_n\widehat{Q}_n^{\prime}(c_nz)}=\frac{1}{\pi}\int_{-1}^{1}\frac{\sqrt{1-t^2}}{z-t}\,dt=z\left(1-\sqrt{1-1/z^2}\right) ,
\end{eqnarray*}
on compact subsets $K\subset \CC \setminus \Delta_c$. From \eqref{OrthDiff_02} and the preceding relations we have for the Laguerre case
\begin{eqnarray}
\label{strongL}\frac{P_n(c_nz)}{\widehat{Q}_n(c_nz)}&=&\frac{zc_n}{\lambda_n}\frac{\widehat{Q}_n^{\prime\prime}(c_nz)}{\widehat{Q}^{\prime}_n(c_nz)}\frac{\widehat{Q}_n^{\prime}(c_nz)}{\widehat{Q}_n(c_nz)}+\left(\frac{1+\alpha-c_nz}{\lambda_n}\right)\frac{\widehat{Q}_n^{\prime}(c_nz)}{\widehat{Q}_n(c_nz)},
\end{eqnarray}
and for the Hermite case
\begin{eqnarray}
\label{strongH}\frac{P_n(c_nz)}{\widehat{Q}_n(c_nz)}&=&\frac{1}{2}\frac{1}{\lambda_n}\frac{\widehat{Q}_n^{\prime\prime}(c_nz)}{\widehat{Q}^{\prime}_n(c_nz)}\frac{\widehat{Q}_n^{\prime}(c_nz)}{\widehat{Q}_n(c_nz)}-\left(\frac{c_nz}{\lambda_n}\right)\frac{\widehat{Q}_n^{\prime}(c_nz)}{\widehat{Q}_n(c_nz)}.
\end{eqnarray}

Taking limits in \eqref{strongL} and  \eqref{strongH} we  obtain  \eqref{StrQ}. Relation \eqref{nrootcon} follows from  \eqref{Pnrootc} and \eqref{StrQ}.
\end{proof}

\section{The polynomial $\mathfrak{Q}_{n}$}\label{withouthat}

Some basic properties of the zeros of  $\mathfrak{Q}_n$  follow directly from \eqref{Lag_DO} and \eqref{Her_DO}. For example, the multiplicity of the zeros of $\mathfrak{Q}_n$ is at most $3$, a zero of multiplicity $3$ is also a zero of $\mathfrak{P}_n$ and a zero of multiplicity $2$ is  a critical point of $\widehat{\mathfrak{Q}}_n$. In the next lemma  we prove conditions for the  boundedness of the zeros of $\mathfrak{Q}_n$ and determine their  asymptotic behavior.

\begin{lemma} \label{Th5LH} Let  $\mu \in \mathcal{P}_m(\Delta)$, where $m \in \NN$ and define for $z\in\CC$, $\dsty\mathfrak{D}(z)=\sup_{x\in \Delta_c}|z-x|$. If $\{\zeta_n\}_{n=m+1}^{\infty}$ is a sequence of complex numbers  with limit  $\zeta \in \CC$, then for every  $\;d>1$ there is a positive number $N_d$, such that $\{z \in \CC : \mathfrak{Q}_n(z)=0 \} \subset \{z \in \CC : |z| \leq \mathfrak{D}(\zeta)+d\}$ whenever $n >N_d$.

\end{lemma}

\begin{proof}\
As  $\mathfrak{Q}_n(z)=0$ then  $\widehat{\mathfrak{Q}}_n(z)=\widehat{\mathfrak{Q}}_n(\zeta_n)$. From Gauss--Lucas theorem (cf. \cite[\S 2.1.3]{She02}), it is known that the critical points of $\widehat{\mathfrak{Q}}_n $ are in the convex hull of its zeros and from $b)$ of Theorem \ref{countingmeasureconv}  the zeros of the polynomials $\{\widehat{\mathfrak{Q}}_n\}_{n=m+1}^{\infty}$  accumulate  on $\Delta_c$. Hence, from the \emph{bisector theorem} (see  \cite[\S 5.5.7]{She02} )  $|z| \leq \mathfrak{D}(\zeta_n)+1$ and the lemma is established.
\end{proof}

We are now ready to prove  Theorem \ref{zeroloc}.

\begin{proof}\emph{[Theorem \ref{zeroloc}]} From Lemma \ref{Th5LH}  we have that the zeros of $\mathfrak{Q}_{n}$ are located in a compact set. From \eqref{(13)LHnor}  the  zeros of $\mathfrak{Q}_n$  satisfy
the equation
\begin{equation}\label{RaizCerosLH}
 \left|\widehat{\mathfrak{Q}}_{n}(z)\right|^{\frac{1}{n}}
 = \left|\widehat{\mathfrak{Q}}_{n}(\zeta_n)\right|^{\frac{1}{n}}\,.
\end{equation}

If $z \in \CC \setminus \Delta_c$, taking limit when  $n \rightarrow \infty$, from Lemma \ref{Th5LH}, and using  \eqref{nrootcon} of Theorem \ref{Th2} on  both sides of  \eqref{RaizCerosLH}, we have that the  zeros of the sequence of polynomials $\{\mathfrak{Q}_n\}_{n=m+1}^{\infty}$    cannot accumulate outside the set

$$ \{z\in \CC:  \Psi(z)=\Psi(\zeta)\}\, \bigcup \, \Delta_c.$$

To verify the second statement of the theorem, note that if $z$ is a zero of $\mathfrak{Q}_n$, from  \eqref{(13)LHnor} we get
\begin{equation}\label{Cero_condLH}
\prod_{k=1}^{n} \left|\frac{z-\widehat{x}_{n,k}}{\zeta_n-\widehat{x}_{n,k}} \right| = 1,\; \mbox{where $\widehat{x}_{n,k}$ are the zeros of $\widehat{\mathfrak{Q}}_n$.}
\end{equation}

Let $\dsty \mathcal{V}_{\varepsilon}(\Delta_c)$ be the $\varepsilon$-neighborhood  of $\Delta_c$ defined as $\dsty \mathcal{V}_{\varepsilon}(\Delta_c)= \{ z \in \CC : \mathfrak{d}(z) < \epsilon  \}$,  as $\dsty \lim_{n \to \infty} \zeta_n=\zeta$, then for all $\varepsilon >0$ there is a $N_{\varepsilon}>0$ such that $|\mathfrak{d}(\zeta_n)-\mathfrak{d}(\zeta)|< \varepsilon $ whenever $n > N_{\varepsilon}$.

If  $\mathfrak{d}(\zeta)>2$, let us choose $\dsty \varepsilon = \varepsilon_{\zeta} = \frac{1}{2} \, \left(\mathfrak{d}(\zeta)-2 \right)$ and suppose that  there is a $z_0 \in \mathcal{V}_{\varepsilon_{\zeta}}(\Delta_c)$  such that $\mathfrak{Q}_n(z_0)=0$ for some $n > N_{\varepsilon_{\zeta}}$. Hence
\begin{equation}\label{Cero_cond}
\prod_{k=1}^{n}\left|\frac{z_0-\widehat{x}_{n,k}}{\zeta_n-\widehat{x}_{n,k}} \right| < \left( \frac{2+\varepsilon_{\zeta}}{\mathfrak{d}(\zeta_n)} \right)^{n}< 1,
\end{equation}
which is a contradiction with \eqref{Cero_condLH}, hence $\{z \in \CC : \mathfrak{Q}_n(z)=0 \} \bigcap \mathcal{V}_{\varepsilon_n}(\Delta_c)= \varnothing$ for all $n >N_{\varepsilon_{\zeta}}$, i.e.  the zeros of $\mathfrak{Q}_n$  cannot accumulate on $\dsty \mathcal{V}_{\varepsilon_{\zeta}}(\Delta_c).$

From \eqref{(13)LHnor} it is straightforward that a multiple zero of $\mathfrak{Q}_n$ is also a critical point of $\widehat{\mathfrak{Q}}_n$. But, from b) of Theorem \ref{countingmeasureconv}  and the Gauss--Lucas theorem  the set of accumulation points of $\widehat{\mathfrak{Q}}_n$ is $\Delta_c$, where we have that for $n$ sufficiently large  the zeros of ${\mathfrak{Q}}_n$ are simple.
\end{proof}

\begin{proof} \emph{[Theorem \ref{RelativeAsymptoticLH}]}

\noindent 1.- Let us prove first that
\begin{equation}\label{AsintComp}
    \frac{\mathfrak{Q}_{n}(z)}{\widehat{\mathfrak{Q}}_n (z)} =1  - \frac{\widehat{\mathfrak{Q}}_n (\zeta_n)}{\widehat{\mathfrak{Q}}_n (z)} \unifn  1,
\end{equation}
uniformly on compact subsets $K$ of the set $\{z \in \CC : |\Psi(z)| > |\Psi(\zeta)| \}$.
In order to prove  \eqref{AsintComp} it is sufficient to show that
\begin{equation}\label{SufAsint}
\frac{\widehat{\mathfrak{Q}}_n (\zeta_n)}{\widehat{\mathfrak{Q}}_n (z)} \unifn 0,
\end{equation}
uniformly on $K$.

From  \cite{GeAssch90} and  Lemmas \ref{LPright}, \ref{LPleft}, we have that for all $n$ large enough it is possible to find constants $c^*,c$ such that
\begin{equation}\label{SemStrongLH}
 c^* \leq \left|\frac{\mathfrak{P}_n(z)}{\Psi^n(z)}\right| \leq  c,
\end{equation}
uniformly on  compact subsets  of  $\CC \setminus \Delta_c$.Then we have
$$
 \left|\frac{\widehat{\mathfrak{Q}}_n (\zeta_n)}{\widehat{\mathfrak{Q}}_n (z)}\right|= \left|\frac{\widehat{\mathfrak{Q}}_n(\zeta_n)}{\mathfrak{P}_n(\zeta_n)}\right|
\left|\frac{\mathfrak{P}_n(z)}{\widehat{\mathfrak{Q}}_n (z)}\right|\left|\frac{\mathfrak{P}_n(\zeta_n)}{\Psi^n(\zeta_n)}\right| \left|\frac{\Psi^n(z)}{\,\mathfrak{P}_n(z)}\right| \left|\left(\frac{\Psi(\zeta_n)}{\Psi(z)}\right)\right|^n.
$$

From \eqref{StrQ} of Theorem \ref{Th2} and \eqref{SemStrongLH} the first four  factors on the right hand side of the previous formula  are bounded; meanwhile, the last factor tends to $0$  when $n \to \infty$, and we get \eqref{SufAsint}. Finally, the assertion 1 is straightforward from \eqref{StrQ} of Theorem \ref{Th2}.

\medskip \noindent 2.- For the assertion 2 of the theorem it is sufficient to prove that
 \begin{equation}\label{AsintComp-2}
    \frac{\mathfrak{Q}_{n}(z)}{\widehat{\mathfrak{Q}}_n (\zeta_n)} =\frac{\widehat{\mathfrak{Q}}_n (z)}{\widehat{\mathfrak{Q}}_n (\zeta_n)}  - 1 \unifn  -1,
\end{equation}
uniformly on compact subsets $K$ of the set $\{z \in \CC : |\Psi(z)| < |\Psi(\zeta)| \}\setminus \Delta_c$. Note that
$$\frac{\widehat{\mathfrak{Q}}_n (z)}{\widehat{\mathfrak{Q}}_n (\zeta_n)} = \frac{\widehat{\mathfrak{Q}}_n (z)}{\mathfrak{P}_n(z)} \frac{\mathfrak{P}_n(\zeta_n)}{\widehat{\mathfrak{Q}}_n(\zeta_n)} \frac{\mathfrak{P}_n(z)} {\Psi^n(z)}
\frac{\Psi^n(\zeta_n)}{\mathfrak{P}_n(\zeta_n)}
 \left(\frac{\Psi(z)}{\Psi(\zeta_n)}\right)^n.
$$
Now,  the first part of the assertion 2 is straightforward from \eqref{StrQ} of Theorem \ref{Th2} and \eqref{SemStrongLH}.

If $\dsty \mathfrak{d}(\zeta)>2$,  let $\dsty \mathcal{V}_{\varepsilon}(\Delta_c)= \{ z \in \CC : \mathfrak{d}(z) < \epsilon  \}$ be a $\varepsilon$--neighborhood  of $\Delta_c$, where  $\dsty \varepsilon = \varepsilon_{\zeta} = \frac{\mathfrak{d}(\zeta)}{2}-1$. By the same reasoning used to deduce \eqref{Cero_cond} we get that
\begin{equation}\label{Normal1}
\left|\frac{\widehat{\mathfrak{Q}}_n (z)}{\widehat{\mathfrak{Q}}_n (\zeta_n)}\right| < \kappa^{n}, \quad \mbox{for all } z \in \mathcal{V}_{\varepsilon}(\Delta_c), \kappa <1.
\end{equation}
Hence from  the first part of the assertion 2 and  \eqref{Normal1} we get the second part of the assertion 2.
\end{proof}

\section{A fluid dynamics model}
\label{Sec_FluidLH}

In this section we show a hydrodynamical model for the zeros of the orthogonal polynomials  with respect to the pair $(\LL,\mu)$. In \cite{BorPij12}, we gave a hydrodynamic interpretation for the critical points  of orthogonal polynomials with respect to a  Jacobi differential operator.

Let us consider a flow of an incompressible fluid in the complex plane, due to a system of  $n$  \emph{different} points ($n>1$) fixed at  $w_i$, $ 1 \leq i \leq n$. At each point $w_i$ of the system there is defined a complex potential $\mathcal{V}_{i}$, which for the Laguerre case equals to the sum of a \emph{source(sink)} with a fixed strength $\Re[c_i]$ plus a \emph{vortex} with a fixed strength $\Im[c_i]$ plus a \emph{uniform stream}  $U_i$ at infinity. Here $c_i$ and $d_i$ are fixed complex numbers which depend on the position of the remaining points $\{w_i\}_{i=1}^{n}$, see \cite[Ch. VIII]{MTho98} for the terminology. We shall call \emph{$n$ system} to the set of the  $n$  points fixed at  $w_i$ with its respective potential of velocities.

Define the functions
$$\dsty f_i(w_1,\ldots, w_n)=\frac{R_n^{\prime \prime}(w_i)}{\dsty R_n^{\prime}(w_i)},\quad i=1,\dots ,n \quad \mbox{where} \quad  R_n(z)=\prod_{i=1}^{n}(z-w_i).$$

The complex potentials $\mathcal{V}_{L}$ (Laguerre case) or $\mathcal{V}_{H}$ (Hermite case) at any point $z$ (see \cite[Ch. 10]{Dur08}),  by the principle of superposition  of solutions, are given by
\begin{equation}
\mathcal{V}_{L}(z) =  \sum_{i=1}^{n}\mathcal{V}_{L,i} = \sum_{i=1}^{n} \left(-z + (1+\alpha-w_i) \log{(z-w_i)}  + (z+w_i \log{(z-w_i)} ) f_i(w_1,\ldots, w_n)\right), \label{(14)LHQn}
\end{equation}
and
\begin{equation}
\label{(14')HQ}\mathcal{V}_{H}(z)= {\sum_{i=1}^{n}\mathcal{V}_{H,i}} = \sum_{i=1}^{n}
\left(-z+ \frac{1}{2}(f_i(w_1,\ldots, w_n)-2w_i)\log{(z-w_i)}\right).
\end{equation}

From a complex potential $\mathcal{V}$, a complex velocity $\mathbf{V}$ can be derived by differentiation ($\mathbf{V}(z)=\frac{d\mathcal{V}}{dz}$). A standard problem associated with the complex velocity is to find the zeros, that correspond to the set of \emph{stagnation points}, i.e. points where the fluid has zero velocity.

We are interested in the problem: Build an $n$ system (location of the points $w_1, \ldots, w_{n}$) such that the stagnation points are at preassigned points with \emph{nice} properties. As it is well known, the zeros of the orthogonal polynomials with respect to a finite positive Borel measures on $\RR$ have a rich set of \emph{nice} properties (cf. in \cite[Ch. VI]{Szg75}),  and will  be taken as preassigned stagnation points. Here we consider $\mu\in \mathcal{P}_1[\RR_+]$  or $\mu\in \mathcal{\widetilde{P}}_2[\RR]$. In the next paragraph we establish the statement of the problem for both Laguerre and Hermite cases.

\begin{quote}
\noindent\textbf{Problem.} \emph{ Let $\{x_1, \ldots, x_n\}$ be the set of zeros of the $n$th orthogonal polynomial $P_n$  ($n>1$ for the Laguerre case and $n>2$ for the Hermite) with respect to a  measure $\mu\in \mathcal{P}_1[\RR_+]$  or $\mu\in \mathcal{\widetilde{P}}_2[\RR]$. Suppose   a flow is given, with complex potential $\mathcal{V}_{L}$ (Laguerre case) or $\mathcal{V}_{H}$ (Hermite). Build an $n$ system (location of the points $w_1, \ldots, w_{n}$) such that the stagnation points  are attained at the points $z=x_i$, with $i=1,2,\ldots,n$.}
\end{quote}

Consider first the Laguerre case. If   $x_{k}$ ($k=1,\ldots, n.$) are  stagnation points then
\begin{equation}\label{(15)LH}
\frac{\partial \mathcal{V}_{L}}{\partial z}(x_k)= (1+\alpha-x_k){\sum_{i=1}^{n} \frac{1}{x_k-w_i}}+
 x_k\sum_{i=1}^{n}\frac{R_n^{\prime \prime}(w_i)}{R_n^{\prime}(w_i)(x_k-w_i)}= 0.
\end{equation}

We are looking for a solution $\dsty R_n(z)=\prod_{i=1}^{n}(z-w_i)$, with $w_i\neq w_j\neq x_k, \forall i,j,k$, $i\neq j$, such that \eqref{(15)LH} holds. This assumption  implies that the sum in the second term of the left hand side of expression \eqref{(15)LH} is the partial-fraction decomposition of $\dsty \frac{R_n^{\prime \prime}}{R_n}$ evaluated at $x=x_k$. Therefore,  \eqref{(15)LH} is equivalent to
$$  x_k R_n^{\prime \prime}(x_k)+(1+\alpha-x_k)R_n^{\prime}(x_k)= 0, \quad
k=1,\,2,\, \dots, n.$$

Note that $xR^{\prime\prime}_n(x)+ (1+\alpha-x)R^{\prime}_n(x)$ is a polynomial of degree $n$, with leading coefficient $\lambda_n$    that vanishes at the zeros of $P_n$, i.e.
\begin{equation}\label{(16L)}
 xR^{\prime\prime}_n(x)+   (1+\alpha-x)R^{\prime}_n(x)= \lambda_n P_n(x).
\end{equation}

Observe  that  expression \eqref{(16L)} is equivalent to \eqref{OrthDiff_02}. From Proposition \ref{Zero_Loc02}, the zeros of $\widehat{Q}_n, \widehat{Q}^{\prime}_n$ are real, simple and $Q^{\prime}_n(x_k)\neq 0$. Therefore, $R_n =\widehat{Q}_n $ is a solution. Hence, an answer  to our problem yields the $n$ points as the $n$ zeros of the polynomial $\widehat{Q}_n $.

For the Hermite case we have a similar situation. Thus, if  $x_{k}$ is a stagnation point,  $ \dsty \frac{\partial \mathcal{V}_{H}}{\partial z}(x_k)=0$,  which gives
\begin{equation}\label{(15')Q}
x_k\sum_{i=1}^{n} {\frac{1}{x_k-w_i}}-\frac{1}{2} \sum_{i=1}^{n}\frac{R_n^{\prime \prime}(w_i)}{R_n^{\prime}(w_i)(x_k-w_i)}= 0, \quad
k=1,\,2,\, \ldots, n.
\end{equation}

Again,  we can deduce that the expression \eqref{(15')Q} equals to $\dsty \frac{1}{2}R_n^{\prime \prime}(x_k)-x_kR_n^{\prime}(x_k) = 0,$ for $k=1, \ldots, n$.

Note that $\frac{1}{2}R^{\prime\prime}_n(x)-xR^{\prime}_n(x)$ is a polynomial of degree $n$, with leading coefficient $\lambda_n$ that   vanishes at the zeros of $P_n$, i.e.
\begin{equation}\label{(16H)}
 \frac{1}{2}R^{\prime\prime}_n(x)-xR^{\prime}_n(x)= \lambda_n P_n(x).
\end{equation}

Therefore, the expression \eqref{(16H)} is equivalent to \eqref{OrthDiff_02}. From Proposition \ref{Zero_Loc02}, the zeros of $\widehat{Q}_n, \widehat{Q}^{\prime}_n$ are real and simple and $Q^{\prime}_n(x_k)\neq 0$, which implies that $R_n =\widehat{Q}_n $ is a solution to our problem. As a conclusion,
\begin{quote}
\noindent\textbf{Answer.} \emph{The flow of an incompressible two--dimensional fluid, due to  $n$  points with complex potential $\mathcal{V}_{L}$ or $ \mathcal{V}_{H} $, located at  the zeros of the $n$th orthogonal polynomial $\widehat{Q}_n$ with respect to  $(\LL, \mu)$, with  $\mu \in \mathcal{P}_1[\RR_+]$  or $\mu\in \mathcal{\widetilde{P}}_2[\RR]$  has its  $n$ stagnation points at the $n$ zeros of the $n$th orthogonal polynomial $\widehat{Q}_n$.}
\end{quote}

It would be interesting to consider the uniqueness of the solution obtained. In other words, what could be said about the solutions of the form  $Q_n(z)= \widehat{Q}_n(z) - \widehat{Q}_n(\zeta_n)$ and to extend this model to more general classes of measures $\mu$. It would be also of interest to decide if these stagnation or equilibrium  points are stable.

\begin{ack*}
The authors thank the comments and suggestions made by  the referees which helped improve the manuscript.
\end{ack*}


\end{document}